\documentclass[reqno]{amsart}
\usepackage[latin1]{inputenc}
\usepackage[pdftex]{graphicx}
\usepackage{graphics}
\usepackage{graphicx}
\usepackage{epstopdf}
\usepackage[pdftex,hyperindex]{hyperref}
\usepackage{floatflt}
\usepackage{multicol}
\usepackage[T1]{fontenc}
\usepackage{textcomp}
\usepackage{latexsym}
\usepackage{expdlist}
\usepackage{vmargin}
\usepackage{booktabs}
\usepackage{placeins}
\usepackage{array}
\usepackage{amsmath}
\usepackage{amssymb}
\usepackage{amsfonts}
\usepackage{verbatim}
\usepackage{array}
\usepackage{framed}
\usepackage{amsthm}
\usepackage{enumerate}
\usepackage{cite}
\input amssym.def
\input amssym
\theoremstyle{plain}
\newtheorem{theorem}{Theorem}[section]
\newtheorem{lemma}{Lemma}[section]
\newtheorem{corollary}{Corollary}[section]
\newtheorem{proposition}{Proposition}[section]
\theoremstyle{remark}

\theoremstyle{definition}

\newcommand\id{\mathrm{id}}

\allowdisplaybreaks[1]

\begin{document}

\title{Harmonic analysis of little $q$-Legendre polynomials}

\author{Stefan Kahler}

\address{Fachgruppe Mathematik, RWTH Aachen University, Pontdriesch 14-16, 52062 Aachen, Germany}

\address{Lehrstuhl A f\"{u}r Mathematik, RWTH Aachen University, 52056 Aachen, Germany}

\address{Department of Mathematics, Chair for Mathematical Modelling, Technical University of Munich, Boltzmannstr. 3, 85747 Garching b. M\"{u}nchen, Germany}

\email{kahler@mathematik.rwth-aachen.de}

\thanks{The research was begun when the author was a guest at Technical University of Munich, and the author gratefully acknowledges support from the graduate program TopMath of the ENB (Elite Network of Bavaria) and the TopMath Graduate Center of TUM Graduate School at Technical University of Munich. The research was continued and completed at RWTH Aachen University. The graphics were made with Maple.}

\date{\today}

\begin{abstract}
Many classes of orthogonal polynomials satisfy a specific linearization property giving rise to a polynomial hypergroup structure, which offers an elegant and fruitful link to Fourier analysis, harmonic analysis and functional analysis. From the opposite point of view, this allows regarding certain Banach algebras as $L^1$-algebras, associated with underlying orthogonal polynomials. The individual behavior strongly depends on these underlying polynomials. We study the little $q$-Legendre polynomials, which are orthogonal with respect to a discrete measure. We will show that their $L^1$-algebras have the property that every element can be approximated by linear combinations of idempotents. This particularly implies that these $L^1$-algebras are weakly amenable (i. e., every bounded derivation into the dual module is an inner derivation), which is known to be shared by any $L^1$-algebra of a locally compact group; in the polynomial hypergroup context, weak amenability is rarely satisfied and of particular interest because it corresponds to a certain property of the derivatives of the underlying polynomials and their (Fourier) expansions w.r.t. the polynomial basis. To our knowledge, the little $q$-Legendre polynomials yield the first example of a polynomial hypergroup whose $L^1$-algebra is weakly amenable and right character amenable but fails to be amenable. As a crucial tool, we establish certain uniform boundedness properties of the characters. Our strategy relies on the Fourier transformation on hypergroups, the Plancherel isomorphism, continued fractions, character estimations and asymptotic behavior.
\end{abstract}

\keywords{Orthogonal polynomials, Fourier analysis on hypergroups, $L^1$-algebras on hypergroups, weak amenability, little $q$-Legendre polynomials, continued fractions}

\subjclass[2020]{Primary 33D45; Secondary 42C10, 43A30, 43A62}

\maketitle

\numberwithin{equation}{section}

\section{Introduction}\label{sec:intro}

\subsection{Motivation}

One of the most famous results of mathematics, the `Banach--Tarski paradox', states that any ball in $d\geq3$ dimensions can be split into a finite number of pieces in such a way that these pieces can be reassembled into two balls of the original size. It is also well-known that there is no analogue for $d\in\{1,2\}$, and the Banach--Tarski paradox heavily relies on the axiom of choice \cite{Wa93}. As soon as one accepts the axiom of choice, the paradox becomes no longer paradoxical at all; nevertheless, its proof can be regarded as an essential motivation for a large and important part of harmonic analysis, dealing with \textit{amenability} and related concepts. Originally formulated for groups, in the 1970s amenability notions were extended to Banach algebras (where the link is provided by the group algebra $L^1(G)$ associated with a locally compact group $G$ \cite{Jo72}).\\

In this paper, we consider Banach algebras which come from orthogonal polynomials and can be regarded as $L^1$-algebras of certain hypergroups. Our focus will lie on aspects concerning orthogonal polynomials and Banach algebras, whereas the hypergroup structure appears in the background; to make the paper more self-contained, and for the convenience of the reader, the relevant basics concerning hypergroups and their Fourier analysis will be recalled below (preknowledge on hypergroups is not required). Moreover, the Banach algebra operations under consideration can be formulated in a rather elementary way via polynomial expansions, avoiding additional hypergroup notation. Roughly speaking, hypergroups generalize locally compact groups by allowing the convolution of two Dirac measures to be a probability measure which satisfies certain compatibility and non-degeneracy properties but does not have to be a Dirac measure again. The precise axioms can be found in \cite{BH95}; they considerably simplify if the hypergroup is discrete \cite{La05}, which will always be the case in this paper.\\

The paper studies the class of little $q$-Legendre polynomials $(R_n(x;q))_{n\in\mathbb{N}_0}$, whose definition will be recalled below and which is interesting for several reasons: on the one hand, the orthogonalization measures are purely discrete, which makes these polynomials rather different from their limiting cases $q\to1$ (which are just the well-known Legendre polynomials). On the other hand, the polynomials (or, more precisely, associated recurrence coefficients and operators) are accompanied by certain convergence and compactness properties. We study the characters $\mathbb{N}_0\rightarrow\mathbb{C}$, $n\mapsto R_n(x;q)$, where $x$ lies in the support of the orthogonalization measure. Using methods from continued fractions, we show that the nontrivial characters satisfy certain uniform boundedness properties. These properties and Fourier analysis on hypergroups will enable us to show that the corresponding $L^1$-algebras are spanned by their idempotents, which has important consequences concerning amenability properties. Besides their role in abstract harmonic analysis, certain amenability properties here correspond to very concrete properties of the derivatives of the polynomials and associated (Fourier) expansions (and are therefore interesting with respect to special functions). The study of amenability properties in the context of (polynomial or different) hypergroups is a fruitful field; recent publications deal with generalizations of Reiter's condition \cite{Ku20}, with F{\o}lner type conditions \cite{Al17}, with amenability vs. harmonic $L^p$-functions \cite{NS20} and with `property A' \cite{TAA22}, for instance. Other recent publications deal with generalized notions of biprojectivity and biflatness in the hypergroup context \cite{EMR19,SRSB22}. A recent paper of the author \cite{Ka21b} studies amenability properties for the large class of associated symmetric Pollaczek polynomials, which are very different from the little $q$-Legendre polynomials.

\subsection{Underlying setting and basics about polynomial hypergroups and their Fourier analysis}

We now give a more detailed description of the underlying setting. Let $a_0>0$, $b_0<1$, $(a_n)_{n\in\mathbb{N}},(c_n)_{n\in\mathbb{N}}\subseteq(0,1)$ and $(b_n)_{n\in\mathbb{N}}\subseteq[0,1)$ satisfy $a_n+b_n+c_n=1\;(n\in\mathbb{N}_0)$, where we set $c_0:=0$, and let the polynomials $(P_n(x))_{n\in\mathbb{N}_0}\subseteq\mathbb{R}[x]$ satisfy the three-term recurrence relation $P_0(x)=1$, $P_1(x)=(x-b_0)/a_0$,
\begin{equation}\label{eq:threetermrec}
P_1(x)P_n(x)=a_n P_{n+1}(x)+b_n P_n(x)+c_n P_{n-1}(x)\;(n\in\mathbb{N}).
\end{equation}
It is obvious that $\mathrm{deg}\;P_n(x)=n$. Due to Favard's theorem and well-known uniqueness results from the theory of orthogonal polynomials \cite{Ch78}, $(P_n(x))_{n\in\mathbb{N}_0}$ is orthogonal w.r.t. a unique probability (Borel) measure $\mu$ on the real line which satisfies $\lvert\mathrm{supp}\;\mu\rvert=\infty$, so
\begin{equation*}
\int_\mathbb{R}\!P_m(x)P_n(x)\,\mathrm{d}\mu(x)\neq0\Leftrightarrow m=n.
\end{equation*}
Moreover, the (Fourier) expansions of $(P_n(x))_{n\in\mathbb{N}_0}$ w.r.t. the bases $\{P_k(x):k\in\mathbb{N}_0\}$ are of the form
\begin{equation*}
P_m(x)P_n(x)=\sum_{k=\lvert m-n\rvert}^{m+n}g(m,n;k)P_k(x)\;(m,n\in\mathbb{N}_0)
\end{equation*}
with $g(m,n;\lvert m-n\rvert),g(m,n;m+n)\neq0$ \cite{La05}. Obviously, $(P_n(x))_{n\in\mathbb{N}_0}$ is normalized by $P_n(1)\equiv1$, so one always has $\sum_{k=\lvert m-n\rvert}^{m+n}g(m,n;k)=1$. It is also well-known that the zeros of the $P_n(x)$ are real and located in the interior of the convex hull of $\mathrm{supp}\;\mu$ \cite{Ch78}.\\

In the following, we always assume that the linearization coefficients $g(m,n;k)$ are nonnegative; according to the literature, we call this condition `property (P)'. Given a concrete orthogonal polynomial sequence, establishing or disproving property (P) can be very challenging (see Gasper's important results concerning Jacobi polynomials \cite{Ga70a,Ga70b} and the author's recent extension to generalized Chebyshev polynomials \cite{Ka21a}, for instance, or the author's recent result concerning the class of associated symmetric Pollaczek polynomials \cite{Ka21b}). Property (P) is crucial for the following link between special functions and harmonic and functional analysis: if one defines a convolution from $\mathbb{N}_0\times\mathbb{N}_0$ into the convex hull of the Dirac functions on $\mathbb{N}_0$ by $(m,n)\mapsto\sum_{k=\lvert m-n\rvert}^{m+n}g(m,n;k)\delta_k$, and if one combines this with the identity as involution $\mathbb{N}_0\rightarrow\mathbb{N}_0$, then $(P_n(x))_{n\in\mathbb{N}_0}$ induces a commutative discrete hypergroup with unit element $0$ on $\mathbb{N}_0$.\footnote{$\delta$ with a subscript shall always denote the corresponding Dirac function.} These `polynomial hypergroups' were introduced by Lasser \cite{La83} and are generally very different from groups or semigroups. Since property (P) is satisfied by many orthogonal polynomial sequences $(P_n(x))_{n\in\mathbb{N}_0}$ (and since the individual behavior strongly depends on $(P_n(x))_{n\in\mathbb{N}_0}$), there is an abundance of examples. Many concepts of harmonic analysis take a rather concrete form. As announced above, we briefly recall some basics and, if not stated otherwise, refer to \cite{La83,La05}.
\begin{itemize}
\item If $f:\mathbb{N}_0\rightarrow\mathbb{C}$ is an arbitrary function and if $n\in\mathbb{N}_0$, then the translation $T_n f:\mathbb{N}_0\rightarrow\mathbb{C}$ of $f$ by $n$ reads
\begin{equation*}
T_n f(m)=\sum_{k=\lvert m-n\rvert}^{m+n}g(m,n;k)f(k).
\end{equation*}
If one normalizes the corresponding Haar measure (cf. equation \eqref{eq:haarbackground} below) in such a way that $\{0\}$ is mapped to $1$, it is just the counting measure on $\mathbb{N}_0$ weighted by the values of the Haar function $h:\mathbb{N}_0\rightarrow[1,\infty)$,
\begin{equation*}
h(n):=\frac{1}{g(n,n;0)}=\frac{1}{\int_\mathbb{R}\!P_n^2(x)\,\mathrm{d}\mu(x)}.
\end{equation*}
$h$ is also given via
\begin{equation*}
h(0)=1,\;h(1)=\frac{1}{c_1},\;h(n+1)=\frac{a_n}{c_{n+1}}h(n)\;(n\in\mathbb{N}).
\end{equation*}
For any $p\in[1,\infty)$, let $\ell^p(h):=\{f:\mathbb{N}_0\rightarrow\mathbb{C}:\left\|f\right\|_p<\infty\}$ with $\left\|f\right\|_p:=\left(\sum_{k=0}^\infty\lvert f(k)\rvert^p h(k)\right)^{1/p}$. Furthermore, let $\ell^\infty(h):=\ell^\infty$. If $p\in[1,\infty]$, $f\in\ell^p(h)$ and $n\in\mathbb{N}_0$, then $T_n f\in\ell^p(h)$. Moreover, if $f\in\ell^1(h)$ and $n\in\mathbb{N}_0$, then
\begin{equation}\label{eq:haarbackground}
\sum_{k=0}^{\infty}T_n f(k)h(k)=\sum_{k=0}^{\infty}f(k)h(k).
\end{equation}
\item For $f\in\ell^p(h)$ and $g\in\ell^q(h)$ with $p\in[1,\infty]$ and $q:=p/(p-1)\in[1,\infty]$, the convolution $f\ast g:\mathbb{N}_0\rightarrow\mathbb{C}$ is defined by
\begin{equation*}
f\ast g(n):=\sum_{k=0}^\infty T_n f(k)g(k)h(k).
\end{equation*}
This convolution is an extension of the hypergroup convolution (see above) and commutes, i.e., one has $f\ast g=g\ast f\in\ell^\infty$ \cite{Ka15,La05}.\footnote{We note that results from \cite{Ka15} can also be found in \cite{Ka16b}.} If $f\in\ell^1(h)$, then $f\ast g\in\ell^q(h)$ with $\left\|f\ast g\right\|_q\leq\left\|f\right\|_1\left\|g\right\|_q$. Together with $\left\|.\right\|_1$, $\ast$ and complex conjugation, the set $\ell^1(h)$ becomes a semisimple commutative Banach $\ast$-algebra with unit $\delta_0$, the `$L^1$-algebra' associated with $(P_n(x))_{n\in\mathbb{N}_0}$. Acting via convolution, $\ell^\infty$ is the dual module of $\ell^1(h)$ \cite{La07,La09b}.\footnote{In the references \cite{La07,La09a,La09b}, which are cited frequently in this paper, the additional assumption $b_0\geq0$ was made, but none of the cited results becomes false if this (meaningless) condition is dropped. Furthermore, the class studied in this paper satisfies $b_0\geq0$ anyway.}
\item Let
\begin{equation*}
\mathcal{X}^b(\mathbb{N}_0):=\left\{z\in\mathbb{C}:\sup_{n\in\mathbb{N}_0}\lvert P_n(z)\rvert<\infty\right\}=\left\{z\in\mathbb{C}:\max_{n\in\mathbb{N}_0}\lvert P_n(z)\rvert=1\right\}
\end{equation*}
(where the latter equality is not obvious but always valid), and let $\widehat{\mathbb{N}_0}:=\mathcal{X}^b(\mathbb{N}_0)\cap\mathbb{R}$. $\mathcal{X}^b(\mathbb{N}_0)$ corresponds to the structure space $\Delta(\ell^1(h))$ via the homeomorphism $\mathcal{X}^b(\mathbb{N}_0)\rightarrow\Delta(\ell^1(h))$, $z\mapsto\varphi_z$, $\varphi_z(f):=\sum_{k=0}^\infty f(k)\overline{P_k(z)}h(k)\;(f\in\ell^1(h))$. In the same way, $\widehat{\mathbb{N}_0}$ is homeomorphic to the Hermitian structure space $\Delta_s(\ell^1(h))$. In particular, $\mathcal{X}^b(\mathbb{N}_0)$ and $\widehat{\mathbb{N}_0}$ are compact subsets of $\mathbb{C}$ and $\mathbb{R}$, respectively, which is a consequence of Gelfand's theory. Moreover, one has
\begin{equation*}
\{1\}\cup\mathrm{supp}\;\mu\subseteq\widehat{\mathbb{N}_0}\subseteq[1-2a_0,1].
\end{equation*}
The character $\alpha_z\in\ell^\infty\backslash\{0\}$ which belongs to $z\in\mathcal{X}^b(\mathbb{N}_0)$ is given by
\begin{equation*}
\alpha_z(n):=P_n(z)\;(n\in\mathbb{N}_0).
\end{equation*}
It is easy to see that
\begin{equation*}
T_m\alpha_z(n)=\alpha_z(m)\alpha_z(n)\;(m,n\in\mathbb{N}_0),
\end{equation*}
and it is clear that $(\lvert\alpha_z(n)\rvert)_{n\in\mathbb{N}_0}$ is bounded by $1$. $\alpha_x\in\ell^\infty\backslash\{0\}$ with $x\in\widehat{\mathbb{N}_0}$ is called a symmetric character.
\item For any function $f\in\ell^1(h)$, the Fourier transform $\widehat{f}:\widehat{\mathbb{N}_0}\rightarrow\mathbb{C}$ is given by
\begin{equation*}
\widehat{f}(x)=\sum_{k=0}^\infty f(k)P_k(x)h(k).
\end{equation*}
The Fourier transform $\widehat{f}$ is continuous and satisfies $\left\|\widehat{f}\right\|_\infty\leq\left\|f\right\|_1$, and the Fourier transformation $(\ell^1(h),\left\|.\right\|_1)\rightarrow(\mathcal{C}(\widehat{\mathbb{N}_0}),\left\|.\right\|_{\infty})$, $f\mapsto\widehat{f}$ is continuous and injective. Furthermore, one has $\widehat{f\ast g}=\widehat{f}\;\widehat{g}\;(f,g\in\ell^1(h))$. By the Plancherel--Levitan theorem, one has
\begin{equation*}
\left\|f\right\|_2^2=\left\|\widehat{f}\right\|_2^2=\int_\mathbb{R}\!(\widehat{f}(x))^2\,\mathrm{d}\mu(x)\;(f\in\ell^1(h)).
\end{equation*}
Finally, there is a (unique) isometric `Plancherel isomorphism' $\mathcal{P}:(\ell^2(h),\left\|.\right\|_2)\rightarrow(L^2(\mathbb{R},\mu),\left\|.\right\|_2)$ extending the Fourier transformation, i.e., $\widehat{f}=\mathcal{P}(f)\;(f\in\ell^1(h))$, where equality is meant in the $L^2(\mathbb{R},\mu)$-sense. The corresponding `Plancherel measure' is just the orthogonalization measure $\mu$, and the inverse of $\mathcal{P}$ reads
\begin{equation*}
\mathcal{P}^{-1}(F)(k)=\int_\mathbb{R}\!F(x)P_k(x)\,\mathrm{d}\mu(x)\;(F\in L^2(\mathbb{R},\mu),k\in\mathbb{N}_0).
\end{equation*}
\end{itemize}
Following Nevai \cite{Ne79}, we say that $\mu\in M(1,0)$ if $\lim_{n\to\infty}a_n=\lim_{n\to\infty}c_n=0$ and $\lim_{n\to\infty}b_n=1$.\footnote{We mention that the definition given in \cite{Ne79} is in terms of the recurrence coefficients which are associated with the \textit{orthonormal} version of $(P_n(x))_{n\in\mathbb{N}_0}$; however, relating these recurrence coefficients to $(a_n)_{n\in\mathbb{N}_0}$, $(b_n)_{n\in\mathbb{N}_0}$ and $(c_n)_{n\in\mathbb{N}}$ \cite[(6.8)]{La05} we easily see that Nevai's definition is equivalent to ours.}

\subsection{Little q-Legendre polynomials and outline of the paper}

We are interested in the following class: let $q\in(0,1)$. The sequence $(P_n(x))_{n\in\mathbb{N}_0}=:(R_n(x;q))_{n\in\mathbb{N}_0}$ of little $q$-Legendre polynomials which corresponds to $q$ is given by
\begin{align*} a_0=\frac{1}{q+1},\;a_n&=q^n\frac{(1+q)(1-q^{n+1})}{(1-q^{2n+1})(1+q^{n+1})}\;(n\in\mathbb{N}),\\
c_n&=q^n\frac{(1+q)(1-q^n)}{(1-q^{2n+1})(1+q^n)}\;(n\in\mathbb{N}),\\
b_n&\equiv1-a_n-c_n=\begin{cases} \frac{q}{q+1}, & n=0, \\ \frac{(1-q^n)(1-q^{n+1})}{(1+q^n)(1+q^{n+1})}, & \mbox{else} \end{cases}
\end{align*}
or, equivalently, via the normalization $R_n(1;q)\equiv1$ and the discrete orthogonalization measure
\begin{equation}\label{eq:mulittleqleg}
\mu(\{x\})=\begin{cases} q^n(1-q), & x=1-q^n\;\mbox{with}\;n\in\mathbb{N}_0, \\ 0, & \mbox{else}. \end{cases}
\end{equation}
There is also a basic hypergeometric representation, reading
\begin{equation*}
R_n(x;q)={}_2\phi_1\left(\left.\begin{matrix}q^{-n},q^{n+1} \\ q\end{matrix}\right\rvert q,q-q x\right)=\sum_{k=0}^n\frac{(q^{-n};q)_k(q^{n+1};q)_k}{(q;q)_k}\frac{(q-q x)^k}{(q;q)_k}\;(n\in\mathbb{N}_0).
\end{equation*}
Recall that $(a;q)_0=1$, $(a;q)_n=\prod_{k=1}^n(1-a q^{k-1})\;(n\in\mathbb{N})$ and $(a;q)_{\infty}=\prod_{k=1}^{\infty}(1-a q^{k-1})$. Property (P) is always satisfied, i.e., $(R_n(x;q))_{n\in\mathbb{N}_0}$ induces a polynomial hypergroup on $\mathbb{N}_0$. The Haar weights $h(n)$ are of exponential growth and satisfy
\begin{equation}\label{eq:hlittleqleg}
h(n)=\frac{1}{q^n}\frac{1-q^{2n+1}}{1-q}\;(n\in\mathbb{N}_0),
\end{equation}
and one has $\mathcal{X}^b(\mathbb{N}_0)=\widehat{\mathbb{N}_0}=\mathrm{supp}\;\mu=\{1\}\cup\{1-q^n:n\in\mathbb{N}_0\}$ (so the Fourier transforms of elements of $\ell^1(h)$ are defined on a countable set with cluster point $1$). These basics are taken from \cite{KLS10,La05,La09a}. Property (P) was studied and established by Koornwinder \cite{Ko91a,Ko95}. The hypergroup is of `strong compact type' \cite{FLS05}, which means that the operator $T_n-\id$ is compact on the space $\ell^1(h)$ for all $n\in\mathbb{N}_0$ and which yields that $\alpha_{1-q^n}\in\ell^1(h)$ for every $n\in\mathbb{N}_0$ \cite[Proposition 2]{FLS05}. Clearly, $\mu$ is in the Nevai class $M(1,0)$.\\

It is obvious that $\left\|\alpha_{1-q^n}\right\|_2^2\leq\left\|\alpha_{1-q^n}\right\|_1$ for every $n\in\mathbb{N}_0$. In Section~\ref{sec:uboundedness}, we obtain a uniform boundedness result on the characters and shall see that there is a constant $C>0$ (depending on $q$) such that $\left\|\alpha_{1-q^n}\right\|_1\leq C\left\|\alpha_{1-q^n}\right\|_2^2$ for all $n\in\mathbb{N}_0$ (Theorem~\ref{thm:characternorm}, cf. also Figure~\ref{fig:characternorm}). This will be established via asymptotic behavior and via methods coming from continued fractions, and it will enable us to show that the linear span of the idempotents is dense in $\ell^1(h)$ by applying Fourier analysis on polynomial hypergroups in Section~\ref{sec:idempotents} (Theorem~\ref{thm:wamlittleqleg1}). Hence, we will obtain that $\ell^1(h)$ is weakly amenable (Theorem~\ref{thm:wamlittleqleg2}), which means that there are no nonzero bounded derivations from $\ell^1(h)$ into $\ell^{\infty}$.\footnote{In Section~\ref{sec:amenability}, we will recall amenability notions in more detail, including references.} In contrast to the group case, weak amenability is a rather strong condition on (the $L^1$-algebra of) a polynomial hypergroup. There are characterizations via properties of the derivatives of the polynomials and their (Fourier) expansions w.r.t. the basis $\{P_k(x):k\in\mathbb{N}_0\}$. To our knowledge, the little $q$-Legendre polynomials yield the first example of a polynomial hypergroup whose $L^1$-algebra is both weakly amenable and right character amenable but not amenable. Section~\ref{sec:amenability} is devoted to such amenability considerations. In particular, we will compare the little $q$-Legendre polynomials to certain polynomials for which weak amenability was obtained in a very different way in \cite{Ka15}, as well as to their limiting cases $q\to1$, the Legendre polynomials.

\section{Character estimations, continued fractions and uniform boundedness properties}\label{sec:uboundedness}

The main result of this section is a kind of uniform boundedness property in terms of the norms of the characters.\\

\begin{theorem}\label{thm:characternorm}
Let $q\in(0,1)$ and $P_n(x)=R_n(x;q)\;(n\in\mathbb{N}_0)$. Then there is some $C>0$ such that
\begin{equation}\label{eq:characternorm}
0<\left\|\alpha_{1-q^n}\right\|_2^2<\left\|\alpha_{1-q^n}\right\|_1<C\left\|\alpha_{1-q^n}\right\|_2^2
\end{equation}
for all $n\in\mathbb{N}_0$. It is possible to take
\begin{equation}\label{eq:characternormexplicit}
C=\frac{1}{q^{\left\lceil\frac{\log4}{\log\frac{1}{q}}-1\right\rceil}}\left[\frac{1}{1-q}+\frac{1}{1-q^{2\left\lceil\frac{\log4}{\log\frac{1}{q}}-1\right\rceil+1}}\sum_{k=1}^\infty4^k q^{\frac{\left(2\left\lceil\frac{\log4}{\log\frac{1}{q}}-1\right\rceil+k-1\right)k}{2}}\right].
\end{equation}
\end{theorem}

\begin{figure}
\centering
\begin{minipage}{.5\textwidth}
\centering
\includegraphics[width=1.00\textwidth]{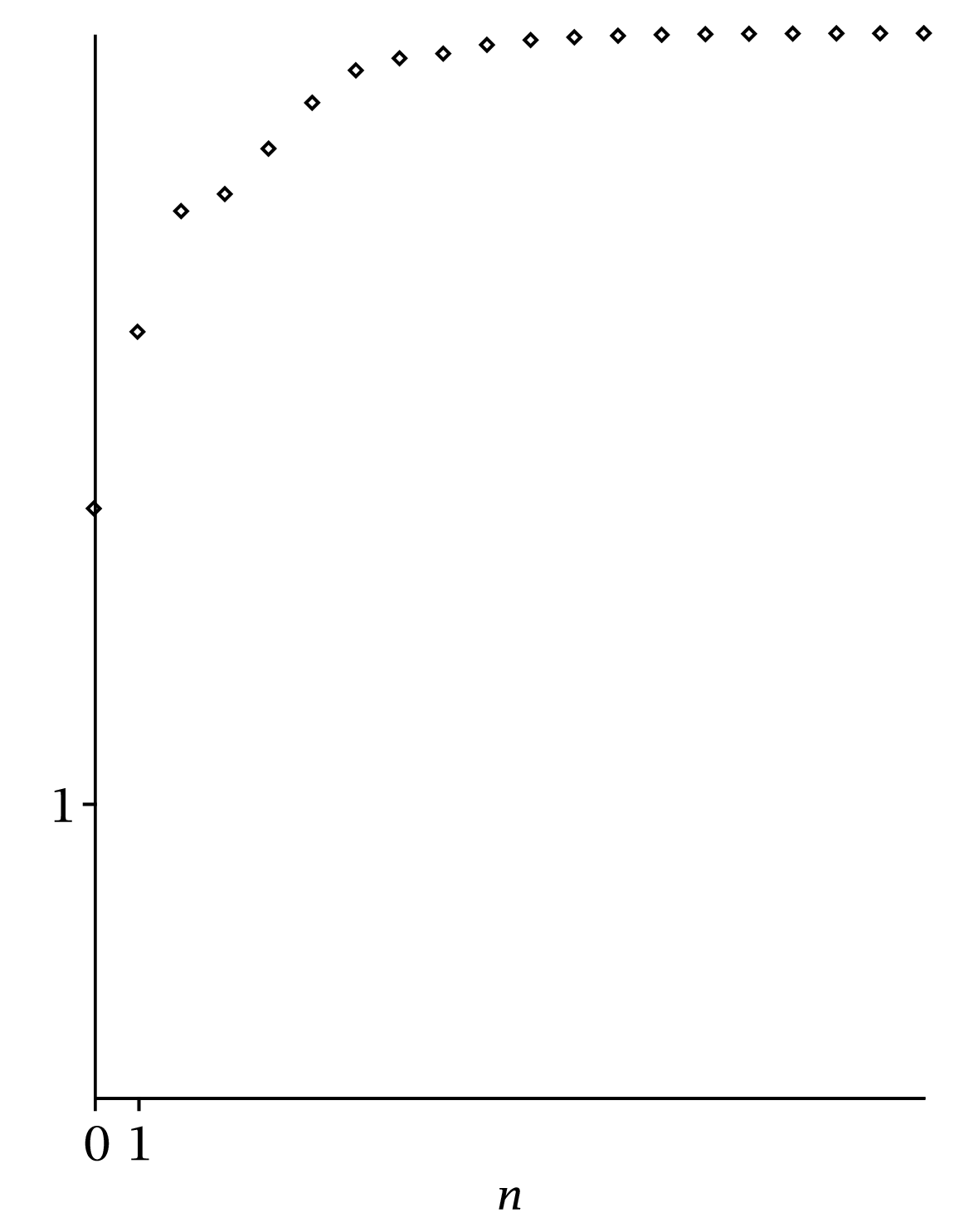}
\end{minipage}%
\begin{minipage}{.5\textwidth}
\centering
\includegraphics[width=1.00\textwidth]{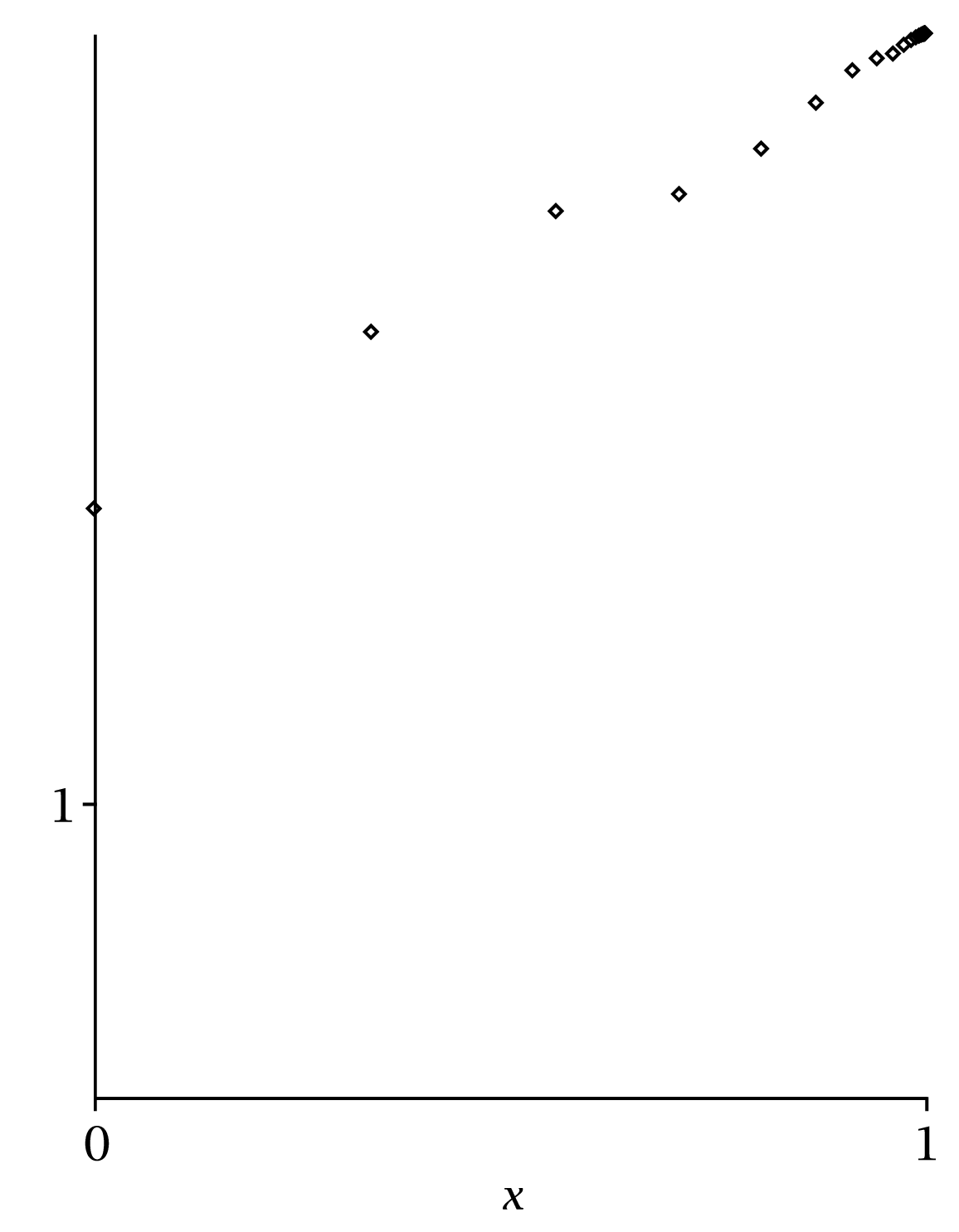}
\end{minipage}
\caption{Left: $\left\|\alpha_{1-q^n}\right\|_1/\left\|\alpha_{1-q^n}\right\|_2^2$ for $n\in\{0,\ldots,19\}$ and $q=2/3$. Right: $\left\|\alpha_x\right\|_1/\left\|\alpha_x\right\|_2^2$ for $x\in\{1-q^n:n\in\{0,\ldots,19\}\}$ and $q=2/3$. The explicit bound provided by Theorem~\ref{thm:characternorm} is $\approx20$.}\label{fig:characternorm}
\end{figure}

Theorem~\ref{thm:characternorm}, which is visualized in Figure~\ref{fig:characternorm} and also essential for Theorem~\ref{thm:wamlittleqleg1} and Theorem~\ref{thm:wamlittleqleg2} in the next sections, relies on another uniform boundedness result given in Lemma~\ref{lma:characterdecay} below. To motivate this crucial lemma, observe that the little $q$-Legendre polynomials $(R_n(x;q))_{n\in\mathbb{N}_0}$ are the little $q$-Jacobi polynomials $(\phi_n^{\alpha,\beta}(1-x))_{n\in\mathbb{N}_0}$ (cf. \cite{IW82}) for $\alpha=\beta=1$; hence, \cite[(1.6)]{IW82} yields the following concerning asymptotics and ratio asymptotics of the characters:

\begin{proposition}\label{prp:asymptotics}
If $q\in(0,1)$ and $P_n(x)=R_n(x;q)\;(n\in\mathbb{N}_0)$, then
\begin{equation*}
\frac{\alpha_{1-q^n}(n+k)}{(-1)^k q^{\frac{k(k+1)}{2}}\frac{(q^{n+1};q)_\infty}{(q;q)_\infty}}\to1\;(k\to\infty)
\end{equation*}
for all $n\in\mathbb{N}_0$. Moreover, for each $n\in\mathbb{N}_0$ the character $\alpha_{1-q^n}$ has at last finitely many zeros and
\begin{equation*}
\frac{\alpha_{1-q^n}(n+k+1)}{\alpha_{1-q^n}(n+k)q^{k+1}}\to-1\;(k\to\infty).
\end{equation*}
\end{proposition}

As a trivial consequence of Proposition~\ref{prp:asymptotics}, for each $n\in\mathbb{N}_0$ there is some $K\in\mathbb{N}_0$ such that $\alpha_{1-q^n}(n+k)\neq0$ and $\left\lvert\alpha_{1-q^n}(n+k+1)/(\alpha_{1-q^n}(n+k)q^{k+1})\right\rvert<4$ for all $k\in\mathbb{N}_0$ with $k\geq K$. The announced lemma improves this by showing that $K$ can be chosen independently of $n$ (so the order of ``$\forall n\in\mathbb{N}_0$'' and ``$\exists K\in\mathbb{N}_0$'' can be changed):

\begin{lemma}\label{lma:characterdecay}
Let $q\in(0,1)$ and $P_n(x)=R_n(x;q)\;(n\in\mathbb{N}_0)$, and let $K\in\mathbb{N}_0$ be given by
\begin{equation*}
K:=\left\lceil\frac{\log4}{\log\frac{1}{q}}-1\right\rceil.
\end{equation*}
Then $\alpha_{1-q^n}(n+k)\neq0$ and
\begin{equation}\label{eq:ratioasymptoticsuni}
\left\lvert\frac{\alpha_{1-q^n}(n+k+1)}{\alpha_{1-q^n}(n+k)q^{k+1}}\right\rvert<4
\end{equation}
for all $n\in\mathbb{N}_0$ and $k\in\mathbb{N}_0$ with $k\geq K$. For every $n\in\mathbb{N}_0$, the sequence $(\alpha_{1-q^n}(n+k))_{k\geq K}$ oscillates around its limit $0$, i.e.,
\begin{equation}\label{eq:oscillating}
\frac{\alpha_{1-q^n}(n+k+1)}{\alpha_{1-q^n}(n+k)}<0\;(k\geq K),
\end{equation}
and the sequence $(\lvert\alpha_{1-q^n}(n+k)\rvert)_{k\geq K}$ is strictly decreasing. Moreover,
\begin{equation}\label{eq:asymptoticsuni}
\lvert\alpha_{1-q^n}(n+K+k)\rvert\leq4^k q^{\frac{(2K+k+1)k}{2}}\lvert\alpha_{1-q^n}(n+K)\rvert\leq4^k q^{\frac{(2K+k+1)k}{2}}
\end{equation}
for all $n\in\mathbb{N}_0$ and $k\in\mathbb{N}_0$.
\end{lemma}

A visualization can be found in Figure~\ref{fig:characterdecay}.\\

\begin{figure}
\centering
\begin{minipage}{.5\textwidth}
\centering
\includegraphics[width=1.00\textwidth]{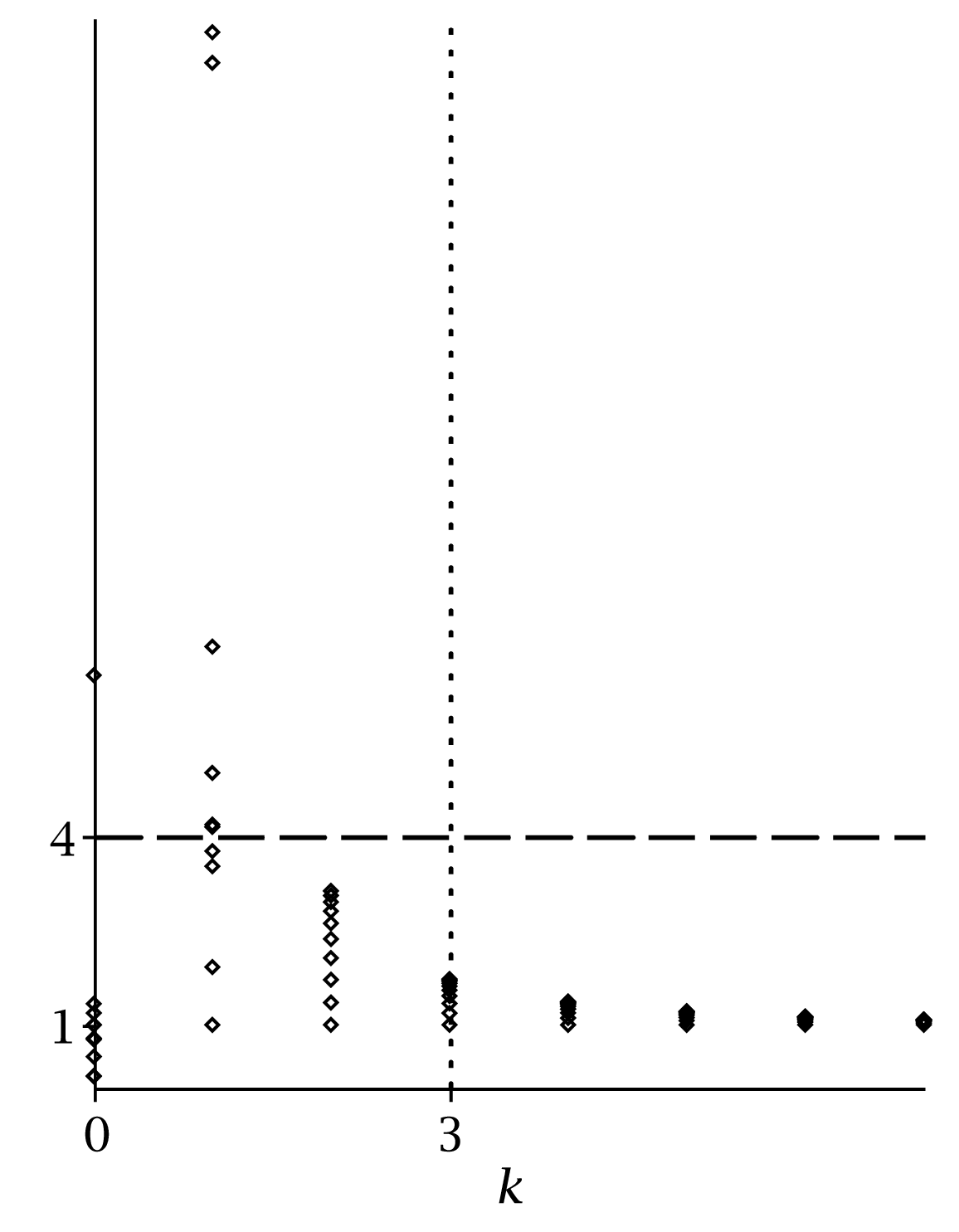}
\end{minipage}%
\begin{minipage}{.5\textwidth}
\centering
\includegraphics[width=1.00\textwidth]{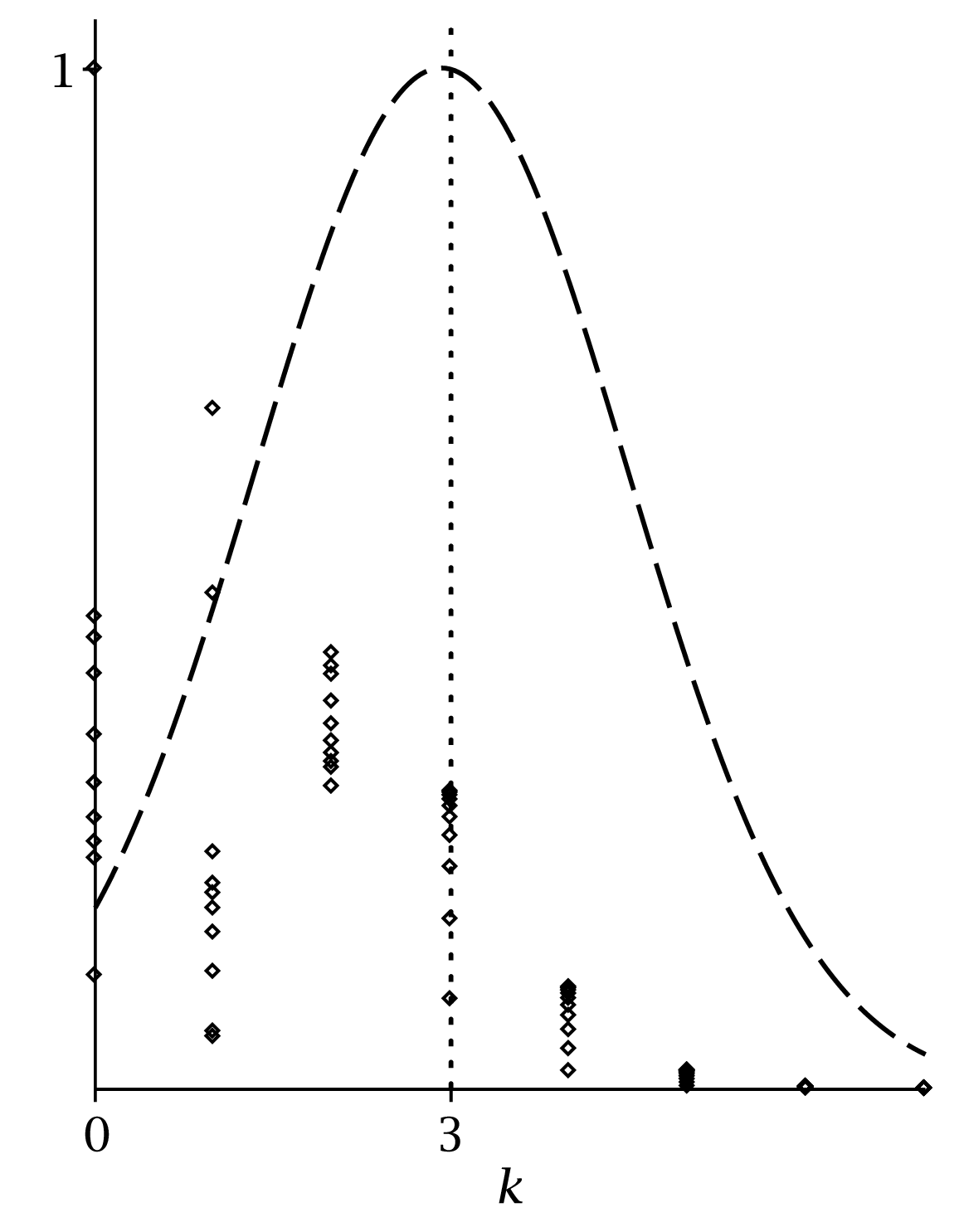}
\end{minipage}
\caption{Left (visualization of \eqref{eq:ratioasymptoticsuni}): $\lvert\alpha_{1-q^n}(n+k+1)/(\alpha_{1-q^n}(n+k)q^{k+1})\rvert$ for $k\in\{0,\ldots,7\}$, $n\in\{0,\ldots,9\}$ and $q=2/3$ (so $K=3$). Right (visualization of \eqref{eq:asymptoticsuni}): $\lvert\alpha_{1-q^n}(n+k)\rvert$ for $k\in\{0,\ldots,7\}$, $n\in\{0,\ldots,9\}$ and $q=2/3$ (so $K=3$); the curve corresponds to the function $x\mapsto4^{x-K}q^{(K+x+1)(x-K)/2}$.}\label{fig:characterdecay}
\end{figure}

The following auxiliary result is needed for the proof of Lemma~\ref{lma:characterdecay}.

\begin{lemma}\label{lma:contfracpre}
Let $q\in(0,1)$ and $P_n(x)=R_n(x;q)\;(n\in\mathbb{N}_0)$. Let
\begin{equation*}
K:=\left\lceil\frac{\log4}{\log\frac{1}{q}}-1\right\rceil
\end{equation*}
and, for all $n,k\in\mathbb{N}_0$,
\begin{equation*}
A_n(k):=\frac{[b_{k+1}-P_1(1-q^n)][b_{k+2}-P_1(1-q^n)]}{a_{k+1}c_{k+2}}
\end{equation*}
and
\begin{equation*}
B_n(k):=\frac{[b_{n+k+1}-P_1(1-q^n)]q^k}{c_{n+k+1}}.
\end{equation*}
Then
\begin{equation}\label{eq:contfracpre1}
A_n(n+k)>4
\end{equation}
and
\begin{equation}\label{eq:contfracpre2}
B_n(k)>\frac{1}{2q}
\end{equation}
for all $n\in\mathbb{N}_0$ and $k\in\mathbb{N}_0$ with $k\geq K$. Moreover,
\begin{equation}\label{eq:contfracpre3}
\lim_{k\to\infty}B_n(k)=\frac{1}{q}.
\end{equation}
\end{lemma}

\begin{proof}
For any $n,k\in\mathbb{N}_0$, we decompose $A_n(k)=C_n(k)D_n(k)$ with
\begin{align*}
C_n(k)&:=\frac{b_{k+1}-P_1(1-q^n)}{a_{k+1}},\\
D_n(k)&:=\frac{b_{k+2}-P_1(1-q^n)}{c_{k+2}}.
\end{align*}
For every $k\in\mathbb{N}_0$, one has
\begin{align*}
b_{k+2}-b_{k+1}&=\frac{(1-q^{k+2})(1-q^{k+3})}{(1+q^{k+2})(1+q^{k+3})}-\frac{(1-q^{k+1})(1-q^{k+2})}{(1+q^{k+1})(1+q^{k+2})}\\
&=\frac{2q^{k+1}(1-q^2)(1-q^{k+2})}{(1+q^{k+1})(1+q^{k+2})(1+q^{k+3})}\\
&>0
\end{align*}
and
\begin{equation*}
\frac{a_{k+1}}{c_{k+2}}-1=\frac{1-q^{2k+5}}{q(1-q^{2k+3})}-1=\frac{(1-q)(1+q^{2k+4})}{q(1-q^{2k+3})}>0.
\end{equation*}
Hence, we see that
\begin{equation*}
C_n(k)<D_n(k)
\end{equation*}
provided $n,k\in\mathbb{N}_0$ are such that $C_n(k)\geq0$; so concerning \eqref{eq:contfracpre1} it suffices to show that $C_n(n+k)\geq2$ for all $n\in\mathbb{N}_0$ and $k\in\mathbb{N}_0$ with $k\geq K$. This is indeed true because, for any $n,k\in\mathbb{N}_0$, a tedious calculation yields
\begin{align*}
&\frac{(1-q^{2n+2k+3})(1+q^{n+k+1})(1+q^{n+k+2})a_{n+k+1}}{q^{2n+k+2}}\left[C_n(n+k)-\left(\frac{1}{q^{k+1}}-2\right)\right]\\
&=2q(1-q^{k+1})+(2+q^{n+k+1}+q^{n+k+2})(1-q^{2n+2k+3})+q^k(2-q^{2n+k+2}-q^{2n+k+4})\\
&>0
\end{align*}
and therefore
\begin{equation*}
C_n(n+k)>\frac{1}{q^{k+1}}-2;
\end{equation*}
by definition of $K$, the right-hand side of the latter inequality is $\geq2$ for all $k\geq K$. By another tedious calculation, we obtain
\begin{align*}
&\frac{(1-q^{2n+2k+2})(1+q^{n+k+2})}{q^{n+k}(1+q^{n+k+1})}\left[B_n(k)-\frac{1-2q^{k+1}}{q}\right]\\
&=2(1-q^{k+1})+q^{n+k+2}(1-q^{2n+2k+3})+q^{k+2}(2-q^{2n+k+1}-q^{2n+k+2})\\
&>0
\end{align*}
and consequently
\begin{equation*}
B_n(k)>\frac{1-2q^{k+1}}{q}
\end{equation*}
for all $n,k\in\mathbb{N}_0$, which implies the second assertion \eqref{eq:contfracpre2} because the right-hand side of the latter inequality is $\geq\frac{1}{2q}$ as soon as $k\geq K$. The calculation above also yields \eqref{eq:contfracpre3}.
\end{proof}

We now come to the proof of Lemma~\ref{lma:characterdecay}. It is based on a ``reverse induction'' argument, on asymptotic behavior and on a suitable application of Worpitzky's theorem \cite{LW92} on continued fractions: let $(s_n)_{n\in\mathbb{N}_0}\subseteq\mathbb{C}$ be a sequence such that
\begin{equation*}
\lvert s_n\rvert\leq\frac{1}{4}
\end{equation*}
for all $n\in\mathbb{N}_0$. Then Worpitzky's theorem states that the continued fraction
\begin{equation*}
\cfrac{1}{1+\cfrac{s_0}{1+\cfrac{s_1}{1+\ddots}}}
\end{equation*}
converges and is an element of the closed disk $\{z\in\mathbb{C}:\lvert z-4/3\rvert\leq2/3\}$.

\begin{proof}[Proof (Lemma~\ref{lma:characterdecay})]
For any $n,k\in\mathbb{N}_0$, let $A_n(k)$ and $B_n(k)$ be defined as in Lemma~\ref{lma:contfracpre}. Let $n\in\mathbb{N}_0$ be fixed. Due to Proposition~\ref{prp:asymptotics}, there is some $M\in\mathbb{N}_0$, $M\geq K$, such that $\alpha_{1-q^n}(n+k)\neq0$ for all $k\in\mathbb{N}_0$ with $k\geq M$. Let $(\phi_{n,k})_{k\geq M}$ be defined by
\begin{equation*}
\phi_{n,k}:=-B_n(k)\frac{\alpha_{1-q^n}(n+k+1)}{\alpha_{1-q^n}(n+k)q^k}=-\frac{[b_{n+k+1}-P_1(1-q^n)]P_{n+k+1}(1-q^n)}{c_{n+k+1}P_{n+k}(1-q^n)}.
\end{equation*}
By Proposition~\ref{prp:asymptotics} and Lemma~\ref{lma:contfracpre}, we have
\begin{equation*}
\lim_{k\to\infty}\lvert\phi_{n,k}\rvert=1,
\end{equation*}
so $M$ can be chosen such that
\begin{equation*}
0<\lvert\phi_{n,k}\rvert<3/2
\end{equation*}
for all $k\in\mathbb{N}_0$ with $k\geq M\geq K$, which shall be assumed from now on. By the recurrence relation \eqref{eq:threetermrec}, we have
\begin{equation}\label{eq:phirec}
\begin{split}
\phi_{n,k+1}&=-\frac{[b_{n+k+2}-P_1(1-q^n)]P_{n+k+2}(1-q^n)}{c_{n+k+2}P_{n+k+1}(1-q^n)}\\
&=-\frac{[b_{n+k+2}-P_1(1-q^n)]\frac{[P_1(1-q^n)-b_{n+k+1}]P_{n+k+1}(1-q^n)-c_{n+k+1}P_{n+k}(1-q^n)}{a_{n+k+1}}}{c_{n+k+2}P_{n+k+1}(1-q^n)}\\
&=\frac{[b_{n+k+1}-P_1(1-q^n)][b_{n+k+2}-P_1(1-q^n)]}{a_{n+k+1}c_{n+k+2}}\\
&\quad\times\left(1+\frac{c_{n+k+1}P_{n+k}(1-q^n)}{[b_{n+k+1}-P_1(1-q^n)]P_{n+k+1}(1-q^n)}\right)\\
&=A_n(n+k)\left(1-\frac{1}{\phi_{n,k}}\right)\;(k\geq M).
\end{split}
\end{equation}
We now define $(\psi_{n,k})_{k\geq K}$ via the continued fractions
\begin{equation*}
\psi_{n,k}:=\cfrac{1}{1-\cfrac{\frac{1}{A_n(n+k)}}{1-\cfrac{\frac{1}{A_n(n+k+1)}}{1-\ddots}}}.
\end{equation*}
Due to Lemma~\ref{lma:contfracpre}, which implies that $0<1/A_n(n+k)<1/4$ for all $k\in\mathbb{N}_0$ with $k\geq K$, and due to Worpitzky's theorem, all of these continued fractions converge and are elements of the interval $[2/3,2]$. In particular, $(\psi_{n,k})_{k\geq K}$ is a sequence of positive reals which is bounded by $2$, and the construction yields
\begin{equation}\label{eq:psirec}
\psi_{n,k+1}=A_n(n+k)\left(1-\frac{1}{\psi_{n,k}}\right)\;(k\geq K).
\end{equation}
Comparing \eqref{eq:phirec} and \eqref{eq:psirec}, we obtain
\begin{equation}\label{eq:characterdecayindeasy}
\lvert\phi_{n,k+1}-\psi_{n,k+1}\rvert=A_n(n+k)\left\lvert\frac{\phi_{n,k}-\psi_{n,k}}{\phi_{n,k}\psi_{n,k}}\right\rvert\;(k\geq M),
\end{equation}
so
\begin{equation*}
\lvert\phi_{n,k+1}-\psi_{n,k+1}\rvert\geq\frac{A_n(n+k)}{3}\lvert\phi_{n,k}-\psi_{n,k}\rvert\;(k\geq M)
\end{equation*}
and consequently
\begin{equation*}
\frac{7}{2}>\lvert\phi_{n,M+k+1}-\psi_{n,M+k+1}\rvert\geq\lvert\phi_{n,M}-\psi_{n,M}\rvert\prod_{j=0}^k\frac{A_n(n+M+j)}{3}\;(k\in\mathbb{N}_0).
\end{equation*}
Since $A_n(n+M+j)/3>4/3$ for all $j\in\mathbb{N}_0$ and consequently
\begin{equation*}
\prod_{j=0}^k\frac{A_n(n+M+j)}{3}\to\infty\;(k\to\infty),
\end{equation*}
this enforces that $\phi_{n,M}=\psi_{n,M}$. We now \textit{claim} that $\alpha_{1-q^n}(n+k)\neq0$ and
\begin{equation*}
\psi_{n,k}=-B_n(k)\frac{\alpha_{1-q^n}(n+k+1)}{\alpha_{1-q^n}(n+k)q^k}
\end{equation*}
for all $k\in\mathbb{N}_0$ with $k\geq K$. Once the claim is proven, we have
\begin{equation}\label{eq:characterdecayprefinish}
\left\lvert B_n(k)\frac{\alpha_{1-q^n}(n+k+1)}{\alpha_{1-q^n}(n+k)q^k}\right\rvert\leq2
\end{equation}
for all $k\in\mathbb{N}_0$ with $k\geq K$; in view of \eqref{eq:characterdecayprefinish}, equation \eqref{eq:ratioasymptoticsuni} then follows with Lemma~\ref{lma:contfracpre}, and \eqref{eq:asymptoticsuni} is immediate from \eqref{eq:ratioasymptoticsuni}. Also \eqref{eq:oscillating} is immediate from Lemma~\ref{lma:contfracpre} and the claim, and the sequence $(\lvert\alpha_{1-q^n}(n+k)\rvert)_{k\geq K}$ is strictly decreasing by \eqref{eq:ratioasymptoticsuni} and the definition of $K$. In view of \eqref{eq:characterdecayindeasy}, the claimed assertion is clear for all $k\in\mathbb{N}_0$ with $k>M$. Hence, we use induction to show that $\alpha_{1-q^n}(n+M-k)\neq0$ and
\begin{equation}\label{eq:characterdecayclaim}
\psi_{n,M-k}=-B_n(M-k)\frac{\alpha_{1-q^n}(n+M-k+1)}{\alpha_{1-q^n}(n+M-k)q^{M-k}}
\end{equation}
for all $k\in\{0,\ldots,M-K\}$. We already know that this is true for $k=0$, so let $k\in\{0,\ldots,M-K\}$ be arbitrary but fixed and assume that $k+1\in\{0,\ldots,M-K\}$, that $\alpha_{1-q^n}(n+M-k)\neq0$ and that \eqref{eq:characterdecayclaim} holds true for $k$, so
\begin{equation*}
\psi_{n,M-k}=-\frac{[b_{n+M-k+1}-P_1(1-q^n)]P_{n+M-k+1}(1-q^n)}{c_{n+M-k+1}P_{n+M-k}(1-q^n)}.
\end{equation*}
Due to \eqref{eq:psirec} and the three-term recurrence relation \eqref{eq:threetermrec}, we have
\begin{align*}
0&<\frac{1}{\psi_{n,M-k-1}}\\
&=1-\frac{\psi_{n,M-k}}{A_n(n+M-k-1)}\\
&=1+\frac{a_{n+M-k}P_{n+M-k+1}(1-q^n)}{[b_{n+M-k}-P_1(1-q^n)]P_{n+M-k}(1-q^n)}\\
&=-\frac{c_{n+M-k}P_{n+M-k-1}(1-q^n)}{[b_{n+M-k}-P_1(1-q^n)]P_{n+M-k}(1-q^n)}\\
&=-\frac{1}{B_n(M-k-1)}\frac{\alpha_{1-q^n}(n+M-k-1)q^{M-k-1}}{\alpha_{1-q^n}(n+M-k)},
\end{align*}
which implies that $\alpha_{1-q^n}(n+M-k-1)\neq0$ and that $k+1$ satisfies \eqref{eq:characterdecayclaim}.
\end{proof}

We need two further lemmas:

\begin{lemma}\label{lma:squarenorms}
Let $q\in(0,1)$ and $P_n(x)=R_n(x;q)\;(n\in\mathbb{N}_0)$. For every $n\in\mathbb{N}_0$, we have
\begin{equation*}
\left\|\alpha_{1-q^n}\right\|_2^2=\frac{1}{q^n(1-q)}.
\end{equation*}
\end{lemma}

\begin{proof}
The proof of \cite[Proposition 2.5.1]{BH95} yields
\begin{equation*}
\frac{\widehat{\alpha_{1-q^n}}}{\left\|\alpha_{1-q^n}\right\|_2^2}=\delta_{1-q^n}.
\end{equation*}
Consequently,
\begin{equation*}
\frac{1}{\left\|\alpha_{1-q^n}\right\|_2^2}=\left\|\frac{\widehat{\alpha_{1-q^n}}}{\left\|\alpha_{1-q^n}\right\|_2^2}\right\|_2^2=\int_\mathbb{R}\!\delta_{1-q^n}^2(x)\,\mathrm{d}\mu(x)=\mu(\{1-q^n\})=q^n(1-q)
\end{equation*}
by the Plancherel--Levitan theorem and \eqref{eq:mulittleqleg}.
\end{proof}

\begin{lemma}\label{lma:hestimate}
Let $q\in(0,1)$ and $P_n(x)=R_n(x;q)\;(n\in\mathbb{N}_0)$. Then
\begin{equation*}
\sum_{k=0}^n h(k)=\frac{1}{1-q}\left[1-\frac{q^{n+1}(2-q^n-q^{n+1})}{1-q^{2n+1}}\right]h(n)<\frac{h(n)}{1-q}
\end{equation*}
for all $n\in\mathbb{N}_0$.
\end{lemma}

\begin{proof}
\eqref{eq:hlittleqleg} and a simple induction on $n$ show the equality, and the inequality is clear.
\end{proof}

\begin{proof}[Proof (Theorem~\ref{thm:characternorm})]
The first inequality in \eqref{eq:characternorm} is clear. The second inequality is clear with ``$\leq$'' for all $n\in\mathbb{N}_0$, and equality can never hold for otherwise one would have $\alpha_{1-q^n}(k)\in\{0,1\}$ for all $k\in\mathbb{N}_0$, contradicting Proposition~\ref{prp:asymptotics}. Let $n\in\mathbb{N}_0$, and let
\begin{equation*}
K:=\left\lceil\frac{\log4}{\log\frac{1}{q}}-1\right\rceil.
\end{equation*}
It is obvious from \eqref{eq:hlittleqleg} that
\begin{equation}\label{eq:hestimate}
\frac{h(m+k)}{h(m)}<\frac{1}{q^k}\frac{1}{1-q^{2m+1}}\;(m,k\in\mathbb{N}_0).
\end{equation}
Using \eqref{eq:hestimate}, we decompose and estimate
\begin{align*}
\left\|\alpha_{1-q^n}\right\|_1&=\sum_{k=0}^\infty\lvert\alpha_{1-q^n}(k)\rvert h(k)\\
&=\sum_{k=0}^{n+K}\lvert\alpha_{1-q^n}(k)\rvert h(k)+\sum_{k=1}^\infty\lvert\alpha_{1-q^n}(n+K+k)\rvert h(n+K+k)\\
&\leq\sum_{k=0}^{n+K}h(k)+\frac{h(n+K)}{1-q^{2n+2K+1}}\sum_{k=1}^\infty\frac{\lvert\alpha_{1-q^n}(n+K+k)\rvert}{q^{k}}\\
&\leq\sum_{k=0}^{n+K}h(k)+\frac{h(n+K)}{1-q^{2K+1}}\sum_{k=1}^\infty\frac{\lvert\alpha_{1-q^n}(n+K+k)\rvert}{q^{k}}.
\end{align*}
Applying Lemma~\ref{lma:characterdecay} and Lemma~\ref{lma:hestimate}, we obtain
\begin{equation}\label{eq:characternormmain}
\left\|\alpha_{1-q^n}\right\|_1<h(n+K)\underbrace{\left[\frac{1}{1-q}+\frac{1}{1-q^{2K+1}}\sum_{k=1}^\infty4^k q^{\frac{(2K+k-1)k}{2}}\right]}_{=:\widetilde{C}};
\end{equation}
note that the series in \eqref{eq:characternormmain} is convergent in $\mathbb{R}$. Finally, Lemma~\ref{lma:squarenorms}, \eqref{eq:hestimate} and \eqref{eq:hlittleqleg} yield
\begin{equation*}
\frac{\left\|\alpha_{1-q^n}\right\|_1}{\left\|\alpha_{1-q^n}\right\|_2^2}<(1-q)\widetilde{C}h(n+K)q^n<\frac{1-q}{1-q^{2K+1}}\widetilde{C}h(K)=\frac{1}{q^K}\widetilde{C}
\end{equation*}
and we obtain the explicit bound
\begin{equation*}
\frac{1}{q^K}\left[\frac{1}{1-q}+\frac{1}{1-q^{2K+1}}\sum_{k=1}^\infty4^k q^{\frac{(2K+k-1)k}{2}}\right]
\end{equation*}
for $\{\left\|\alpha_{1-q^n}\right\|_1/\left\|\alpha_{1-q^n}\right\|_2^2:n\in\mathbb{N}_0\}$, which establishes \eqref{eq:characternormexplicit}.
\end{proof}

\section{Density of idempotents}\label{sec:idempotents}

In this short section, we exploit Theorem~\ref{thm:characternorm} to establish the density of the linear span of the idempotent elements in the Banach algebra $\ell^1(h)$.

\begin{theorem}\label{thm:wamlittleqleg1}
Let $q\in(0,1)$ and $P_n(x)=R_n(x;q)\;(n\in\mathbb{N}_0)$. Then $\ell^1(h)$ is spanned by its idempotents (in the sense that the linear span of the idempotents is dense in $\ell^1(h)$).
\end{theorem}

Before coming to the proof of Theorem~\ref{thm:wamlittleqleg1}, we provide a characterization of the idempotents of $\ell^1(h)$ ($q\in(0,1)$, $P_n(x)=R_n(x;q)$). In the following, let
\begin{equation*}
\epsilon_n:=\frac{1}{h(n)}\delta_n=\mathcal{P}^{-1}(P_n)\;(n\in\mathbb{N}_0).
\end{equation*}
It has already been observed in \cite[Proposition 2.5.1]{BH95} that $\alpha_{1-q^n}/\left\|\alpha_{1-q^n}\right\|_2^2$ is an idempotent and
\begin{equation}\label{eq:characterfourier}
\frac{\widehat{\alpha_{1-q^n}}}{\left\|\alpha_{1-q^n}\right\|_2^2}=\delta_{1-q^n}
\end{equation}
for every $n\in\mathbb{N}_0$ (cf. the proof of Lemma~\ref{lma:squarenorms}). Let $f\in\ell^1(h)$ be an idempotent. Then, for each $n\in\mathbb{N}_0$, $\widehat{f}(1-q^n)\in\{0,1\}$; moreover, $\widehat{f}(1)\in\{0,1\}$. We now distinguish two cases: if $\widehat{f}(1)=0$, then the continuity of $\widehat{f}$ implies that there exists an $N\in\mathbb{N}_0$ such that $\widehat{f}(1-q^n)=0$ whenever $n>N$. In the second case, i.e., if $\widehat{f}(1)=1$, the same argument yields the existence of an $N\in\mathbb{N}_0$ such that $\widehat{f}(1-q^n)=1$ whenever $n>N$. Therefore, due to \eqref{eq:characterfourier} and the injectivity of the Fourier transformation, $f$ is of the form
\begin{equation*}
f=\sum_{n=0}^N\lambda_n\frac{\alpha_{1-q^n}}{\left\|\alpha_{1-q^n}\right\|_2^2}
\end{equation*}
or
\begin{equation*}
f=\epsilon_0-\sum_{n=0}^N\lambda_n\frac{\alpha_{1-q^n}}{\left\|\alpha_{1-q^n}\right\|_2^2},
\end{equation*}
where $N\in\mathbb{N}_0$ and $\lambda_0,\ldots,\lambda_N\in\{0,1\}$. Furthermore, every $f\in\ell^1(h)$ which is of this form is an idempotent, which is a consequence of Shilov's idempotent theorem or can be seen more elementarily from \eqref{eq:characterfourier} and particularly the fact that two idempotents $\alpha_{1-q^m}/\left\|\alpha_{1-q^m}\right\|_2^2,\alpha_{1-q^n}/\left\|\alpha_{1-q^n}\right\|_2^2$, $m,n\in\mathbb{N}_0$ with $m\neq n$, are orthogonal, i.e.,
\begin{equation*}
\frac{\alpha_{1-q^m}}{\left\|\alpha_{1-q^m}\right\|_2^2}\ast\frac{\alpha_{1-q^n}}{\left\|\alpha_{1-q^n}\right\|_2^2}=0.
\end{equation*}

\begin{proof}[Proof (Theorem~\ref{thm:wamlittleqleg1})]
Let $k\in\mathbb{N}$ and
\begin{equation}\label{eq:defidem}
f_k:=\sum_{n=0}^\infty\widehat{\epsilon_0-\epsilon_k}(1-q^n)\frac{\alpha_{1-q^n}}{\left\|\alpha_{1-q^n}\right\|_2^2}.
\end{equation}
For each $n\in\mathbb{N}_0$, one has
\begin{align*}
\lvert\widehat{\epsilon_0-\epsilon_k}(1-q^n)\rvert&=\lvert P_0(1-q^n)-P_k(1-q^n)\rvert\\
&=\lvert1-P_k(1-q^n)\rvert\\
&=\lvert P_k(1)-P_k(1-q^n)\rvert\\
&\leq\max_{x\in[0,1]}\lvert P_k^\prime(x)\rvert q^n
\end{align*}
by the mean value theorem. Thus, in view of Theorem~\ref{thm:characternorm}, the series on the right-hand side of \eqref{eq:defidem} is absolutely convergent in $\ell^1(h)$ (and $\left\|f_k\right\|_1<C/(1-q)\cdot\max_{x\in[0,1]}\lvert P_k^\prime(x)\rvert$, where $C>0$ is as in Theorem~\ref{thm:characternorm}). It is obvious from \eqref{eq:characterfourier} and the continuity of the Fourier transformation that $\widehat{f_k}=\widehat{\epsilon_0-\epsilon_k}$. Therefore, we obtain
\begin{equation*}
\epsilon_k=\epsilon_0-f_k
\end{equation*}
from the injectivity of the Fourier transformation and have shown that $\epsilon_k$ is in the $\left\|.\right\|_1$-closure of the linear span of the idempotents of $\ell^1(h)$. Since the linear span of $\{\epsilon_k:k\in\mathbb{N}_0\}$ is dense in $\ell^1(h)$, this yields the assertion.
\end{proof}

\section{Amenability properties}\label{sec:amenability}

Let us recall some basics concerning amenability properties in a Banach algebraic context: let $A$ be a Banach algebra, and let $D$ be a linear mapping from $A$ into a Banach $A$-bimodule $X$ (i.e., a Banach space which is also an $A$-bimodule and acts continuously \cite{Da00}). Moreover, let $\varphi$ be an element of the structure space $\Delta(A)$. $D$ is called
\begin{itemize}
\item derivation if $D(ab)=a\cdot D(b)+D(a)\cdot b\;(a,b\in A)$,
\item inner derivation if $D(a)=a\cdot x-x\cdot a\;(a\in A)$ for some $x\in X$,
\item point derivation at $\varphi$ if $X=\mathbb{C}$ and $D(a b)=\varphi(a)D(b)+\varphi(b)D(a)\;(a,b\in A)$
\end{itemize}
\cite{Da00}. $A$ is called
\begin{itemize}
\item amenable if for every Banach $A$-bimodule $X$ every bounded derivation into the dual module $X^\ast$ is an inner derivation \cite{Jo72},
\item weakly amenable if every bounded derivation into $A^\ast$ is an inner derivation \cite{Jo88},
\item $\varphi$-amenable if for every Banach $A$-bimodule $X$ such that $a\cdot x=\varphi(a)x\;(a\in A,x\in X)$ every bounded derivation from $A$ into the dual module $X^\ast$ is an inner derivation \cite{KLP08a},
\item right character amenable if $A$ is $\varphi$-amenable for every $\varphi\in\Delta(A)$ and $A$ has a bounded right approximate identity \cite{KLP08b,Mo08}.
\end{itemize}
It is well-known that if there is some $\varphi\in\Delta(A)$ such that a nonzero bounded point derivation exists at $\varphi$, then $A$ is neither $\varphi$-amenable \cite[Remark 2.4]{KLP08a} nor weakly amenable \cite[Theorem 2.8.63]{Da00}. If $A$ is commutative, then $A$ is weakly amenable if and only if there exists no nonzero bounded derivation from $A$ into the dual module $A^\ast$ \cite{BCD87}. The connection to amenability in the group sense is as follows: for any locally compact group $G$, $G$ is amenable if and only if the group algebra $L^1(G)$ is amenable \cite{Jo72}. It was shown in \cite{KLP08b} that right character amenability of $L^1(G)$ characterizes amenability of $G$, too. However, nonzero bounded point derivations can never exist on $L^1(G)$; even more, $L^1(G)$ is always weakly amenable \cite{Jo91}.\\

If one directly generalizes amenability in the group sense (i.e., the existence of a left-invariant mean on $L^{\infty}(G)$) to polynomial hypergroups, one obtains a property which is always satisfied due to the commutativity of these hypergroups \cite[Example 3.3 (a)]{Sk92}; similarly, it is well-known that every Abelian locally compact group is amenable. However, concerning general sequences $(P_n(x))_{n\in\mathbb{N}_0}$ as in Section~\ref{sec:intro} (where property (P) is supposed to be satisfied) and the corresponding Banach algebras $\ell^1(h)$, the amenability properties recalled above are rather strong conditions.\\

There are several general results on amenability properties of $\ell^1(h)$ \cite{Ka15,Ka16b,La07,La09c,La09a,La09b,LP10,Pe11,Wo84}. If $h(n)\to\infty\;(n\to\infty)$, then $\ell^1(h)$ is not amenable \cite[Theorem 3]{La07}. If $\mu$ is absolutely continuous (w.r.t. the Lebesgue--Borel measure on $\mathbb{R}$) and the Radon--Nikodym derivative $\mu^\prime$ is absolutely continuous (as a function) on $[\min\mathrm{supp}\;\mu,\max\mathrm{supp}\;\mu]$, then $\ell^1(h)$ is not weakly amenable \cite[Theorem 2.2]{Ka15}. We call $\ell^1(h)$ `point amenable' if there is no $x\in\widehat{\mathbb{N}_0}$ such that there is a nonzero bounded point derivation at $\varphi_x$; in view of the above, point amenability is necessary for both right character amenability and weak amenability.\footnote{We do not consider point derivations w.r.t. $\varphi\in\Delta(\ell^1(h))\backslash\Delta_s(\ell^1(h))$; however, the class of little $q$-Legendre polynomials which is studied in this paper satisfies $\Delta(\ell^1(h))\backslash\Delta_s(\ell^1(h))=\emptyset$ anyway (because $\mathcal{X}^b(\mathbb{N}_0)=\widehat{\mathbb{N}_0}$ for this class, cf. Section~\ref{sec:intro}). `Point amenability' in our sense must not be confused with `pointwise amenability' considered in \cite{DL10}.} For every $x\in\widehat{\mathbb{N}_0}$, there exists a nonzero bounded point derivation at $\varphi_x$ if and only if $\{P_n^{\prime}(x):n\in\mathbb{N}_0\}$ is bounded \cite[Theorem 1]{La09a}. It is also possible to characterize weak amenability by properties of the derivatives $P_n^{\prime}(x)$; if one defines $(\kappa_n)_{n\in\mathbb{N}_0}\subseteq c_{00}$ via $\kappa_0:=0$ and, for $n\in\mathbb{N}$, via the (Fourier) expansions
\begin{equation*}
P_n^\prime(x)=\sum_{k=0}^{n-1}\kappa_n(k)P_k(x)h(k),\;\kappa_n(k):=0\;(k\geq n)
\end{equation*}
w.r.t. the basis $\{P_k(x):k\in\mathbb{N}_0\}$, or if one defines $(\kappa_n)_{n\in\mathbb{N}_0}$ equivalently by
\begin{equation*}
\kappa_n=\mathcal{P}^{-1}(P_n^\prime)\;(n\in\mathbb{N}_0)
\end{equation*}
via the inverse Plancherel isomorphism, then \cite[Theorem 2]{La07} (or \cite[Theorem 2]{La09b}) states the following:

\begin{theorem}\label{thm:lasserweak}
$\ell^1(h)$ is weakly amenable if and only if $\{\left\|\kappa_n\ast\varphi\right\|_\infty:n\in\mathbb{N}_0\}$ is unbounded for all $\varphi\in\ell^\infty\backslash\{0\}$.
\end{theorem}

The coefficients $\kappa_n$ are also interesting with regard to characterization results on specific classes of orthogonal polynomials \cite{Ka16a,LO08}.\footnote{Results which are cited from \cite{Ka16a} can also be found in \cite{Ka16b}.} If one defines $(\epsilon_n)_{n\in\mathbb{N}_0}$ by
\begin{equation*}
\epsilon_n:=\frac{1}{h(n)}\delta_n=\mathcal{P}^{-1}(P_n)
\end{equation*}
again, \cite[Proposition 1]{La07} (or \cite[Proposition 2]{La09b}) yields the following:\footnote{We note at this stage that our sequence $(\kappa_n)_{n\in\mathbb{N}_0}$ coincides with the sequence which was considered originally in \cite{La07,La09b} only up to the constant factor $a_0$; this does not affect the validity of Theorem~\ref{thm:lasserweak} but causes the additional factor $a_0$ in Proposition~\ref{prp:lasserweak}.}

\begin{proposition}\label{prp:lasserweak}
Let $D:\ell^1(h)\rightarrow\ell^\infty$ be a continuous derivation. Then
\begin{equation*}
D(\epsilon_n)=a_0\kappa_n\ast D(\epsilon_1)\;(n\in\mathbb{N}_0).
\end{equation*}
\end{proposition}

Since the $\kappa_n$ and the $g(m,n;k)$ (hence the convolution $\ast$) are often not explicitly available\footnote{In particular, we are not aware of helpful explicit formulas for the class of little $q$-Legendre polynomials.}, it may be very difficult to apply the unboundedness criterion provided by Theorem~\ref{thm:lasserweak}. Moreover, Theorem~\ref{thm:lasserweak} involves the whole space $\ell^\infty$ but many tools of harmonic analysis are restricted to proper subspaces. Based on Theorem~\ref{thm:lasserweak} and considering the \textit{orthonormal} polynomials
\begin{equation*}
p_n(x)=\sqrt{h(n)}P_n(x)\;(n\in\mathbb{N}_0),
\end{equation*}
in \cite[Theorem 2.3]{Ka15} we found a sufficient criterion which involves absolute continuity w.r.t. the Lebesgue--Borel measure on $\mathbb{R}$:
\begin{theorem}\label{thm:wamsuff}
If each of the conditions
\begin{enumerate}[(i)]
\item $\{\left\|\kappa_n\ast\varphi\right\|_\infty:n\in\mathbb{N}_0\}$ is unbounded for all $\varphi\in\ell^\infty\backslash\mathcal{O}(n^{-1})$,
\item $\mu$ is absolutely continuous, $\mathrm{supp}\;\mu=[-1,1]$, $\mu^\prime>0$ a.e. in $[-1,1]$,
\item $h(n)=\mathcal{O}(n^\alpha)$ (as $n\to\infty$) for some $\alpha\in[0,1)$,
\item $\sup_{n\in\mathbb{N}_0}\int_\mathbb{R}\!p_n^4(x)\,\mathrm{d}\mu(x)<\infty$
\end{enumerate}
holds, then $\ell^1(h)$ is weakly amenable.
\end{theorem}
Applying Theorem~\ref{thm:wamsuff}, we obtained examples such that $\ell^1(h)$ is weakly amenable but fails to be right character amenable, cf. \cite[Section 2.5]{Ka16b} (an explicit example will be recalled at the end of the section; see B) below). The related question whether there are examples such that $\ell^1(h)$ is both weakly amenable and right character amenable but not amenable can be answered positively via little $q$-Legendre polynomials: if $P_n(x)=R_n(x;q)\;(n\in\mathbb{N}_0)$ for some $q\in(0,1)$, then $\ell^1(h)$ is not amenable since $h(n)\to\infty\;(n\to\infty)$ \eqref{eq:hlittleqleg}. However, $\ell^1(h)$ is right character amenable \cite[p. 792]{La09b} (and therefore point amenable \cite[Example 3]{La09a}). In \cite{Ka16b}, we conjectured that $\ell^1(h)$ is weakly amenable. It is clear that Theorem~\ref{thm:wamsuff} cannot be applied because both (ii) and (iii) are violated, which is an obvious consequence of \eqref{eq:mulittleqleg} and \eqref{eq:hlittleqleg} (concerning (iv), see Corollary~\ref{cor:wamlittleqleg2} below). However, weak amenability is an immediate consequence of Theorem~\ref{thm:wamlittleqleg1}:

\begin{theorem}\label{thm:wamlittleqleg2}
Let $q\in(0,1)$ and $P_n(x)=R_n(x;q)\;(n\in\mathbb{N}_0)$. Then $\ell^1(h)$ is weakly amenable.
\end{theorem}

\begin{proof}
This follows from Theorem~\ref{thm:wamlittleqleg1} and the fact that every commutative Banach algebra which is spanned by its idempotents is also weakly amenable \cite[Proposition 2.8.72]{Da00}. Alternatively, there is a slightly more straightforward variant which avoids both Theorem~\ref{thm:wamlittleqleg1} and \cite[Proposition 2.8.72]{Da00} but is also based on Theorem~\ref{thm:characternorm}: in a more explicit way than in the proof of Theorem~\ref{thm:wamlittleqleg1} (because an explicit computation of $\widehat{\epsilon_0-\epsilon_1}(1-q^n)$ is simple), we see that
\begin{equation*}
\epsilon_1=\epsilon_0-\sum_{n=0}^\infty\underbrace{\widehat{\epsilon_0-\epsilon_1}(1-q^n)}_{=(q+1)q^n}\frac{\alpha_{1-q^n}}{\left\|\alpha_{1-q^n}\right\|_2^2}
\end{equation*}
is in the $\left\|.\right\|_1$-closure of the linear span of the idempotents of $\ell^1(h)$ (cf. Figure~\ref{fig:approximation}). Now let $D:\ell^1(h)\rightarrow\ell^\infty$ be a continuous derivation. Since $D$ must be zero on the idempotents \cite[Proposition 1.8.2]{Da00}, we first conclude that $D(\epsilon_1)=0$, and then, applying Proposition~\ref{prp:lasserweak}, that $D(\epsilon_n)=0$ for all $n\in\mathbb{N}_0$. Since the linear span of $\{\epsilon_n:n\in\mathbb{N}_0\}$ is dense in $\ell^1(h)$, we get $D=0$. Hence, $\ell^1(h)$ is weakly amenable.
\end{proof}

\begin{figure}
\centering
\includegraphics[width=1.00\textwidth]{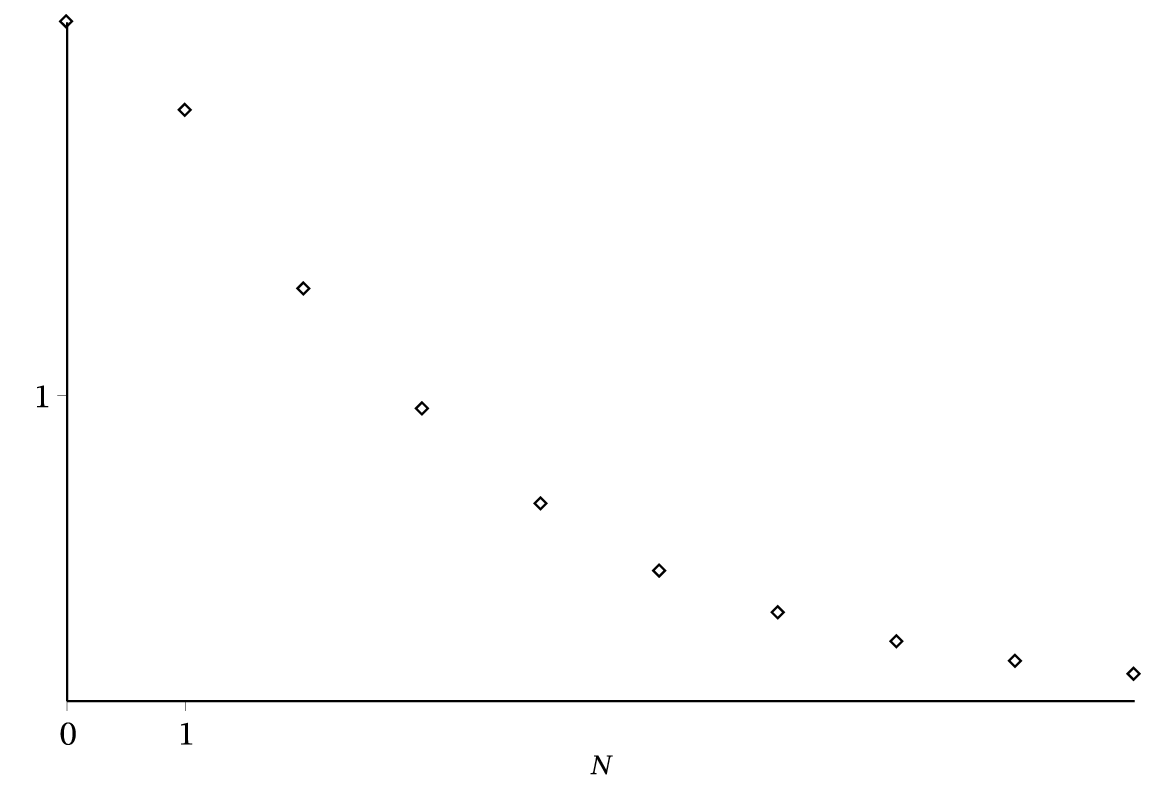}
\caption{$\left\|\epsilon_1-\left(\epsilon_0-\sum_{n=0}^N\widehat{\epsilon_0-\epsilon_1}(1-q^n)\frac{\alpha_{1-q^n}}{\left\|\alpha_{1-q^n}\right\|_2^2}\right)\right\|_1$ for $N\in\{0,\ldots,9\}$ and $q=2/3$.}\label{fig:approximation}
\end{figure}

\begin{corollary}
There exist polynomial hypergroups on $\mathbb{N}_0$ such that $\ell^1(h)$ is weakly amenable and right character amenable but not amenable: for all $q\in(0,1)$, $(R_n(x;q))_{n\in\mathbb{N}_0}$ has the desired properties.
\end{corollary}

\begin{proof}
Immediate from Theorem~\ref{thm:wamlittleqleg2} and the preliminary remarks.
\end{proof}

Theorem~\ref{thm:wamlittleqleg2} is interesting concerning the derivatives of the little $q$-Legendre polynomials. As recalled above, the point amenability (which was already obtained in \cite[Example 3]{La09a}) yields that the set $\{P_n^{\prime}(x):n\in\mathbb{N}_0\}$ is unbounded for every $x\in\{1\}\cup\{1-q^n:n\in\mathbb{N}_0\}$ (provided $P_n(x)=R_n(x;q)\;(n\in\mathbb{N}_0)$, $q\in(0,1)$). In \cite{Ka15}, we (generally) reformulated this property via the sequence $(\kappa_n)_{n\in\mathbb{N}_0}$; applied to the little $q$-Legendre polynomials, we obtain the unboundedness of $\{\left\|\kappa_n\ast\varphi\right\|_\infty:n\in\mathbb{N}_0\}$ for all $\varphi\in\{\alpha_x:x\in\{1\}\cup\{1-q^n:n\in\mathbb{N}_0\}\}\subsetneq\ell^\infty\backslash\{0\}$. If we combine Theorem~\ref{thm:wamlittleqleg2} with Theorem~\ref{thm:lasserweak}, this improves to the following result:

\begin{corollary}
Let $q\in(0,1)$ and $P_n(x)=R_n(x;q)\;(n\in\mathbb{N}_0)$. Then $\{\left\|\kappa_n\ast\varphi\right\|_\infty:n\in\mathbb{N}_0\}$ is unbounded for all $\varphi\in\ell^\infty\backslash\{0\}$.
\end{corollary}

In view of Theorem~\ref{thm:lasserweak}, every polynomial hypergroup with weakly amenable $\ell^1(h)$ must necessarily satisfy condition (i) of our sufficiency criterion Theorem~\ref{thm:wamsuff}. We find the following concerning the remaining conditions of Theorem~\ref{thm:wamsuff}:

\begin{corollary}\label{cor:wamlittleqleg2}
There exist polynomial hypergroups on $\mathbb{N}_0$ such that $\ell^1(h)$ is weakly amenable but all of the conditions (ii), (iii) and (iv) of Theorem~\ref{thm:wamsuff} are violated: for all $q\in(0,1)$, $(R_n(x;q))_{n\in\mathbb{N}_0}$ has the desired properties.
\end{corollary}

\begin{proof}
Let $q\in(0,1)$ and $P_n(x)=R_n(x;q)\;(n\in\mathbb{N}_0)$. It is only left to show that (iv) is violated. Let $(f_n)_{n\in\mathbb{N}_0}\subseteq c_{00}$ be defined by
\begin{equation*}
f_n:=\mathcal{P}^{-1}(p_n^2),
\end{equation*}
and assume that $\sup_{n\in\mathbb{N}_0}\int_\mathbb{R}\!p_n^4(x)\,\mathrm{d}\mu(x)<\infty$. Since
\begin{equation*}
\left\|f_n\right\|_2^2=\int_\mathbb{R}\!p_n^4(x)\,\mathrm{d}\mu(x)
\end{equation*}
(Plancherel--Levitan), the set $\{\left\|f_n\right\|_2^2:n\in\mathbb{N}_0\}$ is bounded, so there has to be a subsequence $(f_{n_j})_{j\in\mathbb{N}_0}$ and a function $f\in\ell^2(h)$ such that $(f_{n_j})_{j\in\mathbb{N}_0}$ is weakly convergent to $f$ in the Hilbert space $\ell^2(h)$. In particular, $(f_{n_j})_{j\in\mathbb{N}_0}$ converges pointwise to $f$. However, by general convergence results on the Nevai class $M(1,0)$ \cite[Lemma 4.2.9, Theorem 4.2.10]{Ne79} we obtain that
\begin{equation*}
f_n(k)=\mathcal{P}^{-1}(p_n^2)(k)=\int_\mathbb{R}\!p_n^2(x)P_k(x)\,\mathrm{d}\mu(x)\to P_k(1)=1\;(n\to\infty)
\end{equation*}
for all $k\in\mathbb{N}_0$, so $(f_n)_{n\in\mathbb{N}_0}$ converges pointwise to the trivial character $\alpha_1\equiv1$. Hence, we get $\alpha_1\in\ell^2(h)$. This is a contradiction, however, because
\begin{equation*}
\left\|\alpha_1\right\|_2^2=\sum_{k=0}^\infty\lvert\alpha_1(k)\rvert^2h(k)=\sum_{k=0}^\infty h(k)=\infty
\end{equation*}
(by the same argument, $\left\|\alpha_1\right\|_2^2=\infty$ holds true for \textit{any} polynomial hypergroup).
\end{proof}

If $q$ is sufficiently small, it can be seen in a more elementary way (avoiding both \cite[Lemma 4.2.9, Theorem 4.2.10]{Ne79} and weak compactness) that condition (iv) of Theorem~\ref{thm:wamsuff} is violated: as a consequence of \cite[(2.7)]{IW82}, one has
\begin{equation}\label{eq:characterlimitingpre1}
\frac{(-1)^n q^{-\frac{n(n+1)}{2}}}{(q;q)_n^2}P_n(1-q^n)=\sum_{k=0}^n\frac{(q^{-n};q)_{n-k}q^{k^2}}{(q;q)_k^2(q;q)_{n-k}^2}\;(n\in\mathbb{N}_0).
\end{equation}
Moreover, \cite[(2.1)]{IW82} gives the transformation
\begin{equation}\label{eq:characterlimitingpre2}
(q^{-n};q)_{n-k}=(-1)^{n-k}q^\frac{(k-n)(n+k+1)}{2}\frac{(q;q)_n}{(q;q)_k}\;(n\in\mathbb{N}_0,k\in\{0,\ldots,n\}).
\end{equation}
Combining \eqref{eq:characterlimitingpre1} and \eqref{eq:characterlimitingpre2}, we get
\begin{equation}\label{eq:characterlimitingpre}
P_n(1-q^n)=\sum_{k=0}^n(-1)^k\frac{(q;q)_n^3q^{\frac{k(3k+1)}{2}}}{(q;q)_k^3(q;q)_{n-k}^2}\;(n\in\mathbb{N}_0).
\end{equation}
Now using the estimation
\begin{equation*}
\frac{(q;q)_n^3q^{\frac{k(3k+1)}{2}}}{(q;q)_k^3(q;q)_{n-k}^2}\leq\frac{q^{\frac{k(3k+1)}{2}}}{(q;q)_{\infty}^5}
\end{equation*}
and applying Lebesgue's dominated convergence theorem to \eqref{eq:characterlimitingpre}, we obtain that $\lim_{n\to\infty}P_n(1-q^n)$ exists with
\begin{equation}\label{eq:characterlimiting}
\lim_{n\to\infty}P_n(1-q^n)=(q;q)_\infty\sum_{k=0}^\infty(-1)^k\underbrace{\frac{q^{\frac{k(3k+1)}{2}}}{(q;q)_k^3}}_{=:\gamma_k}.
\end{equation}
Since
\begin{equation*}
\frac{\gamma_{k+1}}{\gamma_k}=\frac{q^{3k+2}}{(1-q^{k+1})^3}\;(k\in\mathbb{N}_0),
\end{equation*}
$(\gamma_k)_{k\in\mathbb{N}_0}$ is strictly decreasing if $q$ is sufficiently small (which shall be assumed from now on), and \eqref{eq:characterlimiting} and the classical Leibniz criterion imply that
\begin{equation}\label{eq:characterlimitingest}
\lim_{n\to\infty}P_n(1-q^n)\geq(q;q)_\infty(\gamma_0-\gamma_1)>0.
\end{equation}
Since, via \eqref{eq:mulittleqleg} and \eqref{eq:hlittleqleg},
\begin{equation*}
\frac{\int_\mathbb{R}\!p_n^4(x)\,\mathrm{d}\mu(x)}{h(n)}=h(n)\int_\mathbb{R}\!P_n^4(x)\,\mathrm{d}\mu(x)\geq h(n)P_n^4(1-q^n)q^n(1-q)=(1-q^{2n+1})P_n^4(1-q^n),
\end{equation*}
\eqref{eq:characterlimitingest} implies that
\begin{equation*}
\liminf_{n\to\infty}\frac{\int_\mathbb{R}\!p_n^4(x)\,\mathrm{d}\mu(x)}{h(n)}\geq(q;q)_\infty^4(\gamma_0-\gamma_1)^4>0.
\end{equation*}
Since $h(n)\to\infty\;(n\to\infty)$ \eqref{eq:hlittleqleg}, this implies that $\int_\mathbb{R}\!p_n^4(x)\,\mathrm{d}\mu(x)\to\infty\;(n\to\infty)$; hence, condition (iv) of Theorem~\ref{thm:wamsuff} is violated.\\

Besides the non-applicability of Theorem~\ref{thm:wamsuff} to the class of little $q$-Legendre polynomials as a whole, it is interesting to observe that also the principal underlying ideas do not work for this class (which was already observed in \cite{Ka16b}): the proof of Theorem~\ref{thm:wamsuff} given in \cite{Ka15} relies on the sequence $(F_n)_{n\in\mathbb{N}_0}\subseteq c_{00}$,
\begin{equation*}
F_n:=\frac{1}{n+1}\sum_{k=0}^n\mathcal{P}^{-1}(p_k^2),
\end{equation*}
which (under the conditions of Theorem~\ref{thm:wamsuff}) converges in an appropriate sense to a limiting function $F\in\ell^2(h)$ which carries adequate information of the underlying orthogonal polynomial sequence; this convergence arises from an increasingly rapid ``oscillation'' of the polynomials $p_n^2(x)$ around a certain weak limit (as $n$ increases), or, more precisely, due to a strong convergence result for the arithmetic means which can be found in \cite{MNT87}. Moreover, the proof of Theorem~\ref{thm:wamsuff} crucially relies on a density argument concerning the linear span of $\{T_m F:m\in\mathbb{N}_0\}$. However, for the little $q$-Legendre polynomials the sequence $(F_n)_{n\in\mathbb{N}_0}$ converges pointwise to the trivial character $\alpha_1\equiv1$ (cf. the proof of Corollary~\ref{cor:wamlittleqleg2}). This means a ``loss of information'' in two ways: on the one hand, $\alpha_1$ is not specific to the little $q$-Legendre polynomials anymore. On the other hand, the linear span of $\{T_m\alpha_1:m\in\mathbb{N}_0\}=\{\alpha_1\}$ is one-dimensional and therefore inappropriate for analogous density arguments.\\

These considerations show that, despite the weak amenability of $\ell^1(h)$, the harmonic analysis of the little $q$-Legendre polynomials is very different from polynomials which fit in Theorem~\ref{thm:wamsuff}. Concerning the latter, explicit examples are given by certain Jacobi polynomials (see \cite[Section 3]{Ka15}); to illustrate differences to the little $q$-Legendre polynomials in some more detail, we consider the following two specific cases: let $(P_n(x))_{n\in\mathbb{N}_0}\subseteq\mathbb{R}[x]$ be defined by $P_0(x)=1$, $P_1(x)=x$ and
\begin{equation*}
P_1(x)P_n(x)=(1-c_n)P_{n+1}(x)+c_n P_{n-1}(x)\;(n\in\mathbb{N})
\end{equation*}
with either
\begin{enumerate}[A)]
\item $c_n\equiv1/2$ (Chebyshev polynomials of the first kind)
\end{enumerate}
or
\begin{enumerate}[A)]
\item[B)] $c_n\equiv2n/(4n+1)$ (ultraspherical polynomials corresponding to the parameter $-1/4$).
\end{enumerate}
In both cases, the conditions of Theorem~\ref{thm:wamsuff} are fulfilled (see the proof of \cite[Corollary 3.1]{Ka15}) and $\ell^1(h)$ is therefore weakly amenable. However, in A) $\ell^1(h)$ is even amenable \cite[Corollary 3]{La07} \cite[Proposition 5.4.4]{Wo84}, whereas in B) $\ell^1(h)$ fails to be right character amenable \cite[p. 792]{La09b}. Concerning a comparison to the little $q$-Legendre polynomials, recall that the latter yield non-amenable but right character amenable $\ell^1(h)$.\\

Theorem~\ref{thm:wamlittleqleg2} is also interesting when comparing the little $q$-Legendre polynomials to their limiting cases, which are the Legendre polynomials $(\Lambda_n(x))_{n\in\mathbb{N}_0}$. One has
\begin{equation*}
\lim_{q\to1}R_n(x;q)=\Lambda_n(2x-1)\;(n\in\mathbb{N}_0,x\in\mathbb{R}),
\end{equation*}
where $(\Lambda_n(x))_{n\in\mathbb{N}_0}$ is orthogonal w.r.t. the absolutely continuous measure $\chi_{(-1,1)}(x)\,\mathrm{d}x$ and $\Lambda_0(x)=1$, $\Lambda_1(x)=x$ and
\begin{equation*}
\Lambda_1(x)\Lambda_n(x)=\frac{n+1}{2n+1}\Lambda_{n+1}(x)+\frac{n}{2n+1}\Lambda_{n-1}(x)\;(n\in\mathbb{N})
\end{equation*}
\cite{KLS10,La05}. Comparing the $L^1$-algebra which corresponds to $(R_n(x;q))_{n\in\mathbb{N}_0}$ with the $L^1$-algebra which corresponds to $(\Lambda_n(x))_{n\in\mathbb{N}_0}$ (of course, the latter is identical with the $L^1$-algebra that corresponds to $(\Lambda_n(2x-1))_{n\in\mathbb{N}_0}$ because the linearization coefficients $g(m,n;k)$ coincide), one obtains that the behavior w.r.t. point amenability and amenability coincides (see \cite{La07,La09a,La09b} concerning these amenability properties for $(\Lambda_n(x))_{n\in\mathbb{N}_0}$), whereas the behavior w.r.t. weak amenability and right character amenability differs and the two latter properties get lost when passing to the limit $q\to1$ (see \cite{La07,La09c,La09b} concerning $(\Lambda_n(x))_{n\in\mathbb{N}_0}$).

\bibliography{bibliographylittleqlegendre}
\bibliographystyle{amsplain}

\end{document}